\documentclass{amsart}
\usepackage{amsthm,amsfonts,amsmath,amscd,amssymb,latexsym,epsfig,mathrsfs, MnSymbol}
\usepackage{hyperref}
\usepackage[utf8]{inputenc}
\newcommand{\gl}{{\mathfrak g \mathfrak l}}

\newcommand{\ssl}{{\mathfrak s \mathfrak l}}

\newcommand{\g}{{\mathfrak g}}         % Lie algebra of G
\newcommand{\h}{{\mathfrak h}}         % Lie algebra of H
\newcommand{\cx}{{\mathbb C}}

\newcommand{\ad}{\operatorname{ad}}
\newcommand{\Ad}{\operatorname{Ad}}

\newcommand{\Lie}{\operatorname{Lie}}

\newcommand{\Gr}{\operatorname{Gr}}

\newcommand{\Hilb}{\operatorname{Hilb}}

\numberwithin{equation}{section}

\newtheorem{theorem}{Theorem}[section]

\newtheorem{lemma}[theorem]{Lemma}

\newtheorem{proposition}[theorem]{Proposition}

\theoremstyle{remark}

\newtheorem{remark}[theorem]{Remark}

\newtheorem{example}[theorem]{Example}

\newcommand{\oN}{{\mathbb{N}}}

\newcommand{\oP}{{\mathbb{P}}}

\newcommand{\sO}{{\mathcal{O}}}

\newcommand{\sS}{{\mathcal{S}}}

\newcommand{\fH}{{\mathfrak{h}}}

\def\Sym{{\displaystyle\mathfrak{S}}}

\begin{document}

\title{On the Moore-Tachikawa varieties}
\author{Roger Bielawski}
\address{Institut f\"ur Differentialgeometrie\\
Universit\"at Hannover\\ Welfengarten 1\\ D-30167 Hannover}

%\email{R.Bielawski@ed.ac.uk}
%\date{\today}

%\thanks{}
%\dedicatory{}

\subjclass[2020]{14J42, 14L30, 53D20, 57K16}

%\LinearAV
\begin{abstract}
 Moore-Tachikawa varieties are certain Hamiltonian holomorphic symplectic  varieties conjectured in the context of $2$-dimensional topological quantum field theories.  We discuss several constructions related to these varieties. 
\end{abstract}

\maketitle

\thispagestyle{empty}

\section{Introduction}

Let $G_\cx$ be a  simple and simply connected complex Lie group. 
In \cite{MT} Moore and Tachikawa conjectured the existence of a functor $\eta_{G_\cx}$ from the category of $2$-bordisms to the category of holomorphic symplectic varieties with Hamiltonian action, such that gluing of boundaries corresponds to the holomorphic symplectic quotient with respect to the diagonal action $G_\cx$. 
%The properties of the functor  imply that it suffices to define it on basic cobordims with $1,2$ or $3$ incoming circles.
If $W^b_{G_\cx}$, $b\in \oN$, denotes the image of the basic cobordism with $b$ incoming circles, then the complex symplectic varieties are supposed to have the following properties:
\begin{itemize}
\item[(a)] $W^b_{G_\cx}$ has a Hamiltonian action of $\Sym_b\ltimes G_\cx^b$;
\item[(b)] the symplectic quotient of $W^b_{G_\cx}\times W^{b^\prime}_{G_\cx}$,  $b+b^\prime \geq 3$, by $G_\cx$ embedded into $G_\cx^b\times G_\cx^{b^\prime}$ via 

\vspace{-3.5mm}

$$g\mapsto \Bigl((1,\dots,1,g,1,\dots,1),(1,\dots,1,g,1,\dots,1)\Bigr) $$

\vspace{-1.5mm}

is isomorphic to  $W^{b+b^\prime-2}_{G_\cx}$ (for any choice of positions of $g$);

\vspace{1mm}

\item[(c)] $W^2_{G_\cx}\simeq T^\ast G_\cx$ and $W^1_{G_\cx}\simeq G_\cx\times \mathscr{S}_{\g_\cx}$, where $\mathscr{S}_{\g_\cx}$ is a Slodowy slice to the regular nilpotent orbit;
\item[(d)] if $\mu_1,\mu_2,\mu_3$ denote the moment maps for the three $G_\cx$-actions on $W^3_{G_\cx}$, then $P(\mu_1)=P(\mu_2)=P(\mu_3)$ for any invariant polynomial $P\in \cx[\g_\cx^\ast]^{G_\cx}$.
\end{itemize}
%%%%%
We observe that these properties are actually contradictory. First of all, (b) implies that the symplectic quotient of  $W^3_{G_\cx}\times W^{1}_{G_\cx}$ by $G_\cx$ is isomorphic to $W^2_{G_\cx}\simeq T^\ast G_\cx$ as a holomorphic $G_\cx\times G_\cx$-manifold. Property (d) implies now that the two moment maps $\nu_1,\nu_2$ on $T^\ast G_\cx$ also satisfy $P(\nu_1)=P(\nu_2)$ for any $P\in \cx[\g_\cx^\ast]^{G_\cx}$. However, the two moment maps on $T^\ast G_\cx$ in the trivialisation $T^\ast G_\cx\simeq G_\cx\times \g_\cx^\ast$ given by right-invariant forms are  $\nu_1(g,X)=X$, $\nu_2(g,X)=-\ad(g^{-1})X$.
\par
An {\em ad hoc} solution to this contradiction is to modify  (b). Fix a Cartan subalgebra $\fH_\cx$ of $\g_\cx$ and let $\theta_{\fH_\cx}$ be the automorphism of $\g_\cx$ determined by $h\mapsto -h$ on $\fH$ and $\alpha\mapsto -\alpha$ on roots\footnote{Recall that a simple Lie algebra is determined by its root system, so that a vector space isomorphism of Cartan subalgebras of two such algebras $\g,\g^\prime$, which induces an isomorphism of root systems, extends to an isomorphism between $\g$ and $\g^\prime$.}.
For example, if $\g_\cx=\ssl(k,\cx)$ and $\fH_\cx$ consists of diagonal matrices, then $\theta_{\fH_\cx}(A)=-A^T$. 
We denote the corresponding involution on $G_\cx$ by the same symbol. 
\par
Axiom (b) can then be modified as follows:
%%%%%
\begin{itemize}
\item[(b')] the symplectic quotient of $W^b_{G_\cx}\times W^{b^\prime}_{G_\cx}$ by $G_\cx$ embedded into $G_\cx^b\times G_\cx^{b^\prime}$ via

\vspace{-4mm}

$$g\mapsto \Bigl((1,\dots,1,\theta_{\fH_\cx}(g),1,\dots,1),(1,\dots,1,g,1,\dots,1)\Bigr) $$

\vspace{-1.5mm}

is isomorphic to  $W^{b+b^\prime-2}_{G_\cx}$.  \end{itemize}
Simultaneously we replace  the action of $G_\cx\times G_\cx$ on  $W^2_{G_\cx}\simeq T^\ast G_\cx$ given by left and right translations with the action given by ({\em left translations, right translations composed with $\theta_{\fH_\cx}$}).

\medskip

A more elegant solution is to consider {\em oriented} cobordisms\footnote{As far as we understand,  Moore and Tachikawa actually do want a functor from the category of oriented cobordisms.}
and attach a holomorphic symplectic variety  $W^{b,b^\prime}_{G_\cx}$  to any oriented pair of pants $X$  with the boundary consisting of $b$ incoming and $b^\prime$ outgoing circles.
We say that a boundary circle is ``incoming" (resp.\ ``outgoing"), if its orientation together with the inward (resp.\ outward) pointing normal is the orientation of $X$.  Thus for two circles we shall have varieties $W^{2,0}_{G_\cx}$, $W^{1,1}_{G_\cx}$, and $W^{0,2}_{G_\cx}$. They are all isomorphic to each other as holomorphic symplectic varieties, but not as $G_\cx$-varieties. For instance, the isomorphism $\phi$ between $W^{2,0}_{G_\cx}$  and $W^{1,1}_{G_\cx}$ is equivariant with respect to one of the ``incoming" copies of $G_\cx$, but it satisfies  
\begin{equation}\phi(g.m)=\theta_{\h_\cx}(g).\phi(m),\label{theta-G}\end{equation}
for the other incoming $G_\cx$ on $W^{2,0}_{G_\cx}$ and the outgoing $G_\cx$ on $W^{1,1}_{G_\cx}$. We can form a symplectic quotient only by matching an outgoing and an incoming circle.

\medskip

The varieties $W^{b,b^\prime}_{G_\cx}$ should have the following properties:
\begin{itemize}
\item[(A)] $W^{b,b^\prime}_{G_\cx}$ has a Hamiltonian action of $(\Sym_b\times \Sym_b^\prime)\ltimes G_\cx^{b+b^\prime}$;
\item[(B)] the symplectic quotient of $W^{b,b^\prime}_{G_\cx}\times W^{c,c^\prime}_{G_\cx}$,  $b+b^\prime +c+c^\prime\geq 3$, by $G_\cx$ embedded diagonally into one of the $b^\prime$ $G_\cx$-factors on $W^{b,b^\prime}_{G_\cx}$ and into one of the $c$ $G_\cx$-factors of $W^{c,c^\prime}_{G_\cx}$ is isomorphic to $W^{b+c-1,b^\prime+c^\prime-1}_{G_\cx}$;

\vspace{1mm}

\item[(C)] $W^{1,0}_{G_\cx}\simeq G_\cx\times \mathscr{S}_{\g_\cx}$, and $W^{1,1}_{G_\cx}\simeq T^\ast G_\cx$;

\vspace{1mm}

\item[(D)] if $\mu_i$, $i=1,\dots, b$ (resp. $\mu_i^\prime$, $i=1,\dots, b^\prime$)  denote the moment maps for the $b$ factors (resp.\ $b^\prime$ factors), then $P(\mu_i)=P(\mu_j)=P(-\mu_k^\prime)$ for any $i,j,k$ and any invariant polynomial $P\in \cx[\g_\cx^\ast]^{G_\cx}$;
\item[(E)] There exists an isomorphism of holomorphic symplectic varieties $\phi:W^{b,b^\prime}_{G_\cx}\to W^{b-1,b^\prime+1}_{G_\cx}$, which is equivariant with respect to the first $b-1$ $G_\cx$-factors on both varieties, and  with respect to the first $b^\prime$ $G_\cx$-factors on both varieties, while it satisfies \eqref{theta-G} with respect to the $b$-th factor on $W^{b,b^\prime}_{G_\cx}$ and $(b^\prime+1)$-th factor on $W^{b-1,b^\prime+1}_{G_\cx}$.
\end{itemize}
%%%%
\begin{remark} The existence of the action of $\Sym_b\times \Sym_b^\prime$ implies then that there is an isomorphism $\phi_{ij}:W^{b,b^\prime}_{G_\cx}\to W^{b-1,b^\prime+1}_{G_\cx}$ having the properties in (E) for any $i=1,\dots,b$, and $j=1,\dots,b^\prime$.\end{remark}

\medskip

In the present work we construct an open dense subset $U^{b,b^\prime}_{G_\cx}$ of $W^{b,b^\prime}_{G_\cx}$ as a symplectic quotient of the product of $b$ copies of  $W^{1,0}_{G_\cx}$ and $b^\prime$ copies of $W^{0,1}_{G_\cx}$ by an abelian group. This open subset  consists of points of $W^{b,b^\prime}_{G_\cx}$ such that the value of the moment map for each $G_\cx$-action is a regular\footnote{We use ``regular" to mean that the dimension of the centraliser equals the rank of $\g_\cx$.} element of $\g_\cx$. Moreover, the manifolds  $U^{b,b^\prime}_{G_\cx}$  satisfy axioms (A), (B), (D), and (E). In other words, $U^{b,b^\prime}_{G_\cx}$ satisfy all axioms (A)-(E), provided that we replace $W^{1,1}_{G_\cx}$ with its open subset  $G_\cx\times \g_\cx^{\rm reg}$.
\par
In \S 3 we relate $U^{b,b^\prime}_{SL(k,\cx)}$ to Hilbert schemes of points in $\cx\times S_k^{b,b^\prime}$, where $S_k^{b,b^\prime}$ is the manifold of rank one tensors in $\bigl(\cx^k\bigr)^{\otimes b}\otimes\bigl(\cx^k\bigr)^{\ast\otimes b^\prime}$.
 As a byproduct, we find an interesting presymplectic manifold $F_{G_\cx}$ for any simple Lie group $G_\cx$ of type $A_{k-1}$. It  contains $G_\cx\times \mathscr{S}_{\g_\cx}$ as an open dense set and it has a locally free Hamiltonian action of $G_\cx$, such that the moment map induces a bijection between $G_\cx$-orbits in $F_{G_\cx}$ and adjoint orbits in $\ssl(k,\cx)$.
 It seems likely that such an $F_{G_\cx}$  exists for any simple Lie group. Indeed, a Lie-theoretic method of construction is indicated by the description in \S 3.

\medskip

{\em Remark 1.} A very different construction of the Moore-Tachikawa varieties as Poisson varieties satisfying all the axioms has been given by Braverman, Finkelberg, and Nakajima \cite{BFN}   (see also \cite{Ara}). It is expected that the two constructions produce varieties isomorphic on an open dense subset, but this does not seem to be easy to prove.
\par
{\em Remark 2.} After this paper has been posted on the ArXiv, I have received a preprint from David Kazhdan  \cite{GK}. In it he and Victor Ginzburg construct varieties $U^{b,0}$ (and much more) by what is essentially the same method, although the technical details are different.
\par
{\em Remark 3.} My motivation for this work is not so much the Moore-Tachikawa conjecture, as the hyperk\"ahler geometry of  $U^{b,b^\prime}_{G_\cx}$. This will be presented in  a future paper.
\par
{\em Acknowledgement.} I thank Lorenzo Foscolo and Hiraku Nakajima for comments. In particular I thank the former for his suggestion to consider $W^{b,b^\prime}_{G_\cx}$ rather than $W^b_{G_\cx}$ with the above {\em ad hoc} solution.

\section{An open dense subset of  $W^{b,b^\prime}_{G_\cx}$ }

Let $G_\cx$ be a complex semisimple Lie group of rank $r$ with Lie algebra $\g_\cx$. We identify $\g_\cx^\ast$ with $\g_\cx$ using the Killing form $\langle\;,\;\rangle$. Let $\h_\cx$ be a fixed Cartan subalgebra and let 
$$ \g_\cx=\h_\cx \oplus\bigoplus_{\alpha\in \Lambda}\Phi_\alpha$$
be a root space decomposition. Let $\Delta$ be a set of simple roots. Let $\mathscr{S}_{\g_\cx}$ be the Slodowy slice $e+Z(f)$ determined by a principal $\ssl(2,\cx)$-triple $(e,f,h)$ with $e\in \bigoplus_{\alpha\in \Delta}\Phi_{-\alpha}$.

Our starting point is the holomorphic symplectic manifold $W^{1,0}_{G_\cx}\simeq G_C\times \mathscr{S}_{\g_\cx}$. The symplectic form at $(g,X)$ is given by
\begin{equation} 
 -\bigl\langle dX\wedge g^{-1}dg\bigr\rangle+\bigl\langle X,g^{-1}dg\wedge g^{-1}dg\bigr\rangle=\bigl\langle  dgg^{-1}\wedge d\bigl(\Ad(g) X\bigr)\bigr\rangle,\label{form1}\end{equation}
where $\langle \phi\wedge \psi\rangle(u,v)=\langle \phi(u),\psi(v)\rangle-\langle \phi(v),\psi(u)\rangle.$
The action of $G_\cx$ by left translations is Hamiltonian with the moment map given by $\mu(g,X)=\Ad( g) X$. 
\par
We define analogously $W^{0,1}_{G_\cx}\simeq G_C\times \mathscr{S}_{\g_\cx}$, with the symplectic form
$$\bigl\langle  g^{-1}dg\wedge d\bigl(\Ad(g^{-1}) X\bigr)\bigr\rangle$$
The action of $G_\cx$ is now on the right and the moment map is $\mu^\prime (g,X)=-\Ad(g^{-1})X$. In particular, $W^{1,0}_{G_\cx}$ and $W^{0,1}_{G_\cx}$ satisfy (E) with 
$$\phi(X,g)=\bigl(-\Ad(p)\circ \theta_{\h_\cx}(X), p\theta_{\h_\cx}(g)^{-1}p^{-1}\bigr),$$ 
where $p\in G_\cx$ conjugates the opposite Slodowy slice (with $e\in \bigoplus_{\alpha\in \Delta}\Phi_{\alpha}$) to $\mathscr{S}_{\g_\cx}$.

\medskip

$W^{1,0}_{G_\cx}$ and $W^{0,1}_{G_\cx}$ also have Hamiltonian actions of the abelian group $A_{\g_\cx}=\cx[\g_\cx]^{G_\cx}\simeq\cx^r$ defined as follows. Write an  invariant polynomial $P\in \cx[\g_\cx]^{G_\cx}$ as $P(X)=p(X,\dots,X)$ for an invariant multilinear and symmetric form $p$ and define $C_P(X)\in \g_\cx$ via
\begin{equation} \bigl\langle C_P(X),Y\bigr\rangle=(\deg P)\cdot p(X,\dots,X,Y) \quad \forall Y\in \g_\cx.\label{C_P}\end{equation}
The elements $C_P(X)$, $P\in \cx[\g_\cx]^{G_\cx}$, generate the centraliser $Z(X)$ of a regular $X$ (cf.\ \cite[Lemma 6]{SC}).
Now $P\in A_{\g_\cx}$  acts on $W^{1,0}_{G_\cx}$ via
\begin{equation} P.(g,X)= \bigl(g\exp(C_P(X)),X\bigr),\label{P-action}\end{equation}
and on $W^{0,1}_{G_\cx}$ via
\begin{equation}  P.(g,X)= \bigl(\exp(C_P(X))g,X\bigr).\label{P2-action}\end{equation}
%%%%
The moment map $\nu$ for the action of $A_{\g_\cx}$ is easily seen to be in both cases
\begin{equation} \nu(g,X)(P)=P(X),\quad P\in \Lie(A_{\g_\cx})\simeq A_{g_\cx}.\label{nu}\end{equation}
%%%
Moreover,  the actions of $G_\cx$ and $A_{\g_\cx}$ commute and the moment map for the $G_\cx$-action identifies the subset $\g^{\rm reg}_\cx\subset \g_\cx$ of regular elements as the geometric quotient of $W^{1,0}_{G_\cx}$ or  $W^{0,1}_{G_\cx}$ by  $A_{\g_\cx}$.

We shall now want to define, for any $b,b^\prime\geq 0$, a holomorphic symplectic variety $U^{b,b^\prime}_{G_\cx}$ as the symplectic quotient (in appropriate sense) of the product of $b$ copies of $W^{1,0}_{G_\cx}$  and $b^\prime$ copies of $W^{0,1}_{G_\cx}$  by the diagonal action of 
\begin{equation} A_0=\bigl\{(P_1,\dots,P_{b+b^\prime})\in (A_{\g_\cx})^{b+b^\prime}\,;\,\sum_{i=1}^{b+b^\prime} P_i=0\}.\label{A_0}\end{equation}
The level set of the moment map is chosen to be $0$.
\begin{lemma} The $0$-level set of the moment map for the action of $A_0$ on $\bigl(W^{1,0}_{G_\cx}\bigr)^b\times \bigl(W^{0,1}_{G_\cx}\bigr)^{b^\prime}$
is given by 
$$X_1=\dots =X_{b+b^\prime}\in\mathscr{S}_{\g_\cx}.$$\label{simple}
\end{lemma}

\vspace{-5mm}

\begin{proof}Let $P\in \cx[\g_\cx]^{G_\cx}$ and consider a $1$-dimensional subgroup of $A_0$ given by 
$$\{0,\dots,0,zP,0,\dots,0, -zP,0,\dots 0\},$$
where the nonzero entries are in the $i$-th and in the $j$-th place. The moment map is, according to \eqref{nu}, equal to
$$z=(g_i,X_i)_{i=1}^{b+b^\prime}\mapsto P(X_i)-P(X_j),$$ 
and so $z$ belongs to the zero set of the moment map for $A_0$ if and only if $X_1,\dots,X_{b+b^\prime}$ belong to the same adjoint orbit, i.e.\ $X_1=\dots =X_{b+b^\prime}.$
\end{proof}
%%%%
Thus the $0$-level set of the moment map is the manifold $Y=G_\cx^{b+b^\prime}\times \mathscr{S}_{\g_\cx}$. The action of $A_0$ on $Y$ is not, however, proper, due to the fact that the centralizer of a regular element may be disconnected. Fortunately, although $A_0$ is not reductive, the GIT issues are straightforward to resolve in this case.
\par
We consider the closure $R\subset Y\times Y$ of the relation defined by $A_0$. Thus $R$ is the set
\begin{equation} \Bigl\{\bigl(g_i,X\bigr),\bigl(h_i,X\bigr)\in Y\times Y;\;Ad(u_i)X=X,\enskip \prod u_i=1\Bigr\},\label{relation}
\end{equation}
where $u_i=h_i^{-1} g_i$ if $i\leq b$ and $u_i=g_ih_i^{-1}$ if $i>b$.  This is an analytic subset of $Y\times Y$ and the analytic relation $R$ satisfies  the assumptions of a theorem of Grauert (see \cite[Thm.\ 7.1]{VII} and the second paragraph on p.\ 203 there) which guarantee that the quotient of $Y$ by this relation, defined as the ringed space $(Y, \sO(Y))/R$,  is a normal complex space.
We define 
$U^{b,b^\prime}_{G_\cx}$ to be this space. We have:
%%%%%%
%%%%
\begin{proposition} $U^{b,b^\prime}_{G_\cx}$ is  smooth. The action of any $G_\cx$-factor on  $U^{b,b^\prime}_{G_\cx}$ is free and proper, and the quotient is biholomorphic to
$$ Q=\bigl\{(X,y_1,\dots,y_{b+b^\prime -1})\in \mathscr{S}_{G_\cx}\times \g_\cx^{b+b^\prime-1}\;;\; y_i\in O(X),\enskip i=1,\dots,b+b^\prime-1\bigr\},$$
where $O(X)$ is the adjoint orbit of $X$.\label{smooth}
\end{proposition}
\begin{proof} We first prove that the action of any $G_\cx$-factor  is free and proper. This is a purely topological statement. %%We view $U^{b,b^\prime}_{G_\cx}$ as $Y/\sim$, where the relation is the one in the previous lemma.  
Without loss of generality, we can consider the first factor $G_\cx$-factor. Suppose that $h\in G_\cx$ stabilises $[1,g_2,\dots,g_{b+b^\prime},X]\in U^{b,b^\prime}_{G_\cx}$. This means that there are $u_1,\dots,u_{b+b^\prime} \in Z_{G_\cx}(X)$
with $\prod u_i=1$ such that $hu_1=1$, $g_iu_i=g_i$ for $i\leq b$, and $u_ig_i=g_i$ for $i>b$. This implies that $h=1$.
\par
Now observe that a $G_\cx$-orbit over $X\in \mathscr{S}_{G_\cx}$  is the same thing as $(Z_{G_\cx}(X))^{b+b^\prime-1}$-orbit on $G_\cx^{b+b^\prime-1}\times \mathscr{S}_{\g_\cx}$, and therefore the (geometric) quotient is homeomorphic to  $Q$. This is clearly Hausdorff, and hence the $G_\cx$-action is proper. 
It follows that the quotient by $G_\cx$ is biholomorphic to $Q$. $Q$ is smooth, since it is the fibred product of $\g_\cx^{\rm reg} \to \mathscr{S}_{G_\cx}$, which is a submersion. Therefore $U^{b,b^\prime}_{G_\cx}$ is a principal bundle over $Q$, hence smooth.
\end{proof}
%%%%%
The symplectic form on  $\bigl(W^{1,0}_{G_\cx}\bigr)^b\times \bigl(W^{0,1}_{G_\cx}\bigr)^{b^\prime}$ is $A_0$-invariant, and hence it descends to  $U^{b,b^\prime}_{G_\cx}$. It is easily written down. For example, on $U^{b,0}_{G_\cx}$
it is given by:
\begin{equation} -\Bigl\langle dX\wedge \bigl(\sum_{i=1}^b g_i^{-1}dg_i\bigr)\Bigr\rangle+\Bigl\langle X,\sum_{i=1}^b g_i^{-1}dg_i\wedge g_i^{-1}dg_i\Bigr\rangle.\label{symplU}\end{equation}
%%%%
\begin{remark}
The moment maps yield a holomorphic map:
$$\Phi:U^{b,b^\prime}_{G_\cx}\longrightarrow \bigl\{ (X,y_1,\dots,y_{b+b^\prime})\in \mathscr{S}_{G_\cx}\times \g_{\cx}^{b+b^\prime}; \; y_i\in O(X)\bigr\}.
$$
The fibre of $\Phi$ at any point is $Z_{G_\cx}(X)$. 
If we represent a vertical vector field by a fundamental vector field generated by $\rho\in Z_{\g_\cx}(X)$, then the
the symplectic form is $\langle \rho \wedge dX\rangle$ plus the sum of Kostant-Kirillov-Souriau forms on  the first $b$ orbits minus the sum of Kostant-Kirillov-Souriau forms on  the last $b^\prime$ orbits. This follows easily from \eqref{symplU}.\label{fibration}
\end{remark}
%%%
%%
\begin{remark} We can extend the construction to include the case $b=b^\prime=0$. The variety  $W^{0,0}_{G_\cx}\simeq U^{0,0}_{G_\cx}$ is the symplectic quotient of  $W^{1,0}_{G_\cx}\times W^{0,1}_{G_\cx}$ by the diagonal action of $G_\cx$. It follows that $W^{0,0}_{G_\cx}$ is isomorphic to 
$$\bigl\{(g,X)\in G^C\times \mathscr{S}_{\g_\cx}; \Ad(g)X=X\bigr\}.$$
\end{remark}

We now have:
\begin{theorem} \begin{itemize}
\item[(i)] $U^{1,1}_{G_\cx}\simeq G_\cx\times \g_\cx^{\rm reg}\subset W^{1,1}_{G_\cx}$.
\item[(ii)] The manifolds $U^{b,b^\prime}_{G_\cx}$ defined above satisfy axioms (A), (B), (D), and (E).
\end{itemize}\label{UABCDE}
\end{theorem}
\begin{proof} $U^{1,1}_{G_\cx}$ is isomorphic to the quotient of
$$\bigl\{(g_1,g_2,X_1,X_2)\,;\, g_1,g_2\in G_\cx,\; X_1=X_2\in  \mathscr{S}_{\g_\cx}\bigr\}$$
by the relation 
$$(g_1,g_2,X)\sim \bigl(g_1u,u^{-1}g_2,X_1,X_2\bigr), \quad u\in Z_{G_\cx}(X).$$
The map 
$$ (g_1,g_2,X_1,X_2)\mapsto \bigl( g_1g_2, \Ad(g_1) X_1\bigr)$$
is constant on equivalence classes and induces an isomorphism $U^{1,1}_{G_\cx}\simeq G_\cx\times \g_\cx^{\rm reg}$.

We now prove (ii). Axioms (A), (D), and (E) are obvious from the construction. Let us index the 
the $G_\cx$-factors in $U^{b,b^\prime}_{G_\cx}$ by $i,i^\prime$, $i=1,\dots,b$, $i^\prime =1,\dots,b^\prime$, and similarly by $j,j^\prime$ on $U^{c,c^\prime}_{G_\cx}$. We consider the symplectic quotient of $U^{b,b^\prime}_{G_\cx}\times U^{c,c^\prime}_{G_\cx}$ by $G_\cx$ acting diagonally on the $p^\prime$-factor on $U^{b,b^\prime}_{G_\cx}$ and on $q$-th factor on $U^{c,c^\prime}_{G_\cx}$ (and trivially on all the other factors). The moment map at
$$ (m,\tilde m)=\bigl([g_1,\dots,g_b,g_{1^\prime},\dots,g_{b_\prime},X],[h_1,\dots,h_c,h_{1^\prime},\dots,h_{c_\prime},Y]\bigr)\in U^{b,b^\prime}_{G_\cx}\times U^{c,c^\prime}_{G_\cx}$$
is $-\Ad(g_{p^\prime}^{-1})X+\Ad(h_q)Y$, and so the level set of the moment map is given by $X=\Ad(g_{p^\prime}h_q)Y$. Since $X,Y\in \mathscr{S}_{\g_\cx}$, we must have $X=Y$ and $g_{p^\prime}h_q\in Z_{G_\cx}(Y)$. Quotienting by $G_\cx$ allows us to make $g_{p^\prime}=1$, and then $h_q\in Z_{G_\cx}(Y)$. Such an $\tilde m$ is $R$-equivalent to one with $h_q=1$. We send this pair of representatives to 
$$[g_1,\dots,g_b,g_{1^\prime},\dots,\widehat{g_{p^\prime}},\dots, g_{b_\prime},h_1,\dots,\widehat{h_q},\dots, h_c,h_{1^\prime},\dots,h_{c_\prime},Y]\in U^{b+c-1,b^\prime+c^\prime-1}_{G_\cx}.$$
It is straightforward to check that $R$-equivalence classes are mapped to $R$-equivalence classes, and so
 the symplectic quotient is isomorphic to $U^{b+c-1,b^\prime+c^\prime-1}_{G_\cx}$.
 % Suppose then $m\sim \tilde{m}$. Writing  $\tilde{m}$  as $\tilde g_1,\tilde h_1$, etc., and setting $u_i=\tilde{g}_i^{-1} g_i$, $u_{i^\prime}=g_{i^\prime}\tilde{g}_{i^\prime}^{-1}$, $v_j=\tilde{h}_j^{-1} h_j$, $v_{j^\prime}=h_{j^\prime}\tilde{h}_{j^\prime}^{-1}$, we have
% $$m\sim \tilde{m}\iff u_i,u_{i^\prime},v_j,v_{j^\prime}\in Z_{G_\cx}(Y),\enskip \prod u_i\prod u_{i^\prime}=1=\prod v_j\prod v_{j^\prime}.$$
%After acting by $z$ and $z^{-1}$, $z\in Z_{G_\cx}(Y)$ on two neighbouring $G_\cx$-factors, we can assume that, in addition, $u_{p^\prime}=v_q^{-1}$. But then
%$$ \prod u_i\prod_{i\neq p^\prime} u_{i^\prime}\prod_{j\neq q} v_j\prod v_{j^\prime}=u_{p^\prime}^{-1} v_q=1,$$ and, hence the images of $m$ and $\tilde m$ in $U^{b+c-1,b^\prime+c^\prime-1}_{G_\cx}$ are equivalent.
\end{proof}

\begin{remark} We can rewrite the symplectic form \eqref{symplU} as 
\begin{equation} \sum_{i=1}^b\bigl\langle dg_ig_i^{-1} \wedge d(\Ad (g_i)X)  \bigr\rangle.\label{symplW}\end{equation}
The $\g_\cx$-valued functions $\Ad (g_i)X$ are the moment maps for the $G_\cx$-factors acting on $U^{b,0}_{G_\cx}$.
Thus, $W^{b,0}_{G_\cx}$ (and similarly $W^{b,b^\prime}_{G_\cx}$) must be an extension of $U^{b,0}_{G_\cx}$ on which this form remains nondegenerate, but the values of the moment maps are no longer required to be regular. Moreover, for $b=2$ the two moment maps are adjoints of each other, but for $b\geq 3$ this no longer can be the case, if axioms (B) and (C) are to be satisfied.
\label{weird}
\end{remark}

\section{The case $G_\cx=SL(k,\cx)$ }

We shall now discuss in detail the $A_{k-1}$-case. It is actually simpler to describe the case when $G_\cx=GL(k,\cx)$;  the 
manifolds $U^{b,b^\prime}_{G_\cx}$ for a simple $G_\cx$ of type $A_{k-1}$ can be then obtained by taking submanifolds and finite quotients.  One reason why $GL(k,\cx)$ is simpler is that the centraliser of any regular element is connected, so that there is no necessity to replace the group $A_0$ by the relation \eqref{relation}. 
\par
We therefore fix $G_\cx=GL(k,\cx)$ and write $U^{b,b^\prime}_k$, $W^{b,b^\prime}_k$, for the corresponding varieties.
Let $S_k^{b,b^\prime}$ be the manifold of rank $1$ tensors in $\bigl(\cx^k\bigr)^{\otimes b}\otimes\bigl(\cx^k\bigr)^{\ast\otimes b^\prime}$. It is biholomorphic to $(\cx^k\backslash\{0\})^{b+b^\prime}/T_0$, where 
$$T_0=\Bigl\{(\lambda_1,\dots,\lambda_{b+b^\prime})\in (\cx^\ast)^{b+b^\prime}\,;\, \prod_{i=1}^b\lambda_i\prod_{i=b+1}^{b+b^\prime}\lambda_i^{-1}=1\Bigr\}.$$
Consider now $X_{k}^{b,b^\prime}=S_k^{b,b^\prime}\times \cx$ and denote by $\pi$ the projection onto the second factor. We recall \cite[Ch.\ 6]{AH} the notion of a {\em transverse} Hilbert scheme of points. If $\pi:Z\to X$ is a surjective holomorphic map between complex manifolds or varieties, then the transverse Hilbert scheme $\Hilb^k_\pi(Z)$ of $k$ points in $Z$ is an open subset of $\Hilb^k(Z)$ consisting of those $D$ for which $\pi|_D$ is a scheme-theoretic isomorphism onto its image.

The $GL(k,\cx)^{b+b^\prime}$-action on $S_k^{b,b^\prime}$ induces a corresponding action on $\Hilb^k(X_{k}^{b,b^\prime})$. 
We shall say that a $0$-dimensional subscheme $D$ of  $X_{k}^{b,b^\prime}$ is {\em nondegenerate} if the stabiliser of $D$ in each factor   of $\prod_{i=1}^{b+b^\prime} GL(k,\cx)$  is trivial.

\begin{theorem} $U^{b,b^\prime}_k$ is equivariantly biholomorphic to an open subset of the transverse Hilbert scheme $\Hilb^k_\pi(X_{k}^{b,b^\prime})$ consisting of nondegenerate subschemes. \label{tHsS}
\end{theorem}
\begin{proof} 
We shall construct an isomorphism from $U^{\rm non-deg}\subset \Hilb^k_\pi(X_{k}^{b,b^\prime})$ to $U^{b,b^\prime}_k$, where $U^{\rm non-deg}$ is the subset of nondegenerate subschemes.
For the sake of transparency of the argument,  we shall assume that $b^\prime=0$; the modifications needed for the general case are obvious. A point $D$ of $\Hilb^k_\pi(X_{k}^{b,0})$ consists of a divisor $\sum k_iz_i$ in $\cx$, $\sum k_i=k$, together with a $(k_i-1)$-jet of a section of $\pi$ at each $z_i$. We consider first the case when $\pi(D)=k0$. Then $D$ is a $(k-1)$-jet of a section of $\pi$ at $0$, i.e.\ the truncation $p(\epsilon)$ of $x_1(\epsilon)\otimes \cdots \otimes x_b(\epsilon)$ to order $k-1$, where $x_j(\epsilon)=\sum_{m=0}^{k-1} x_j^m\epsilon^{m-1}$,  $x_j^m\in \cx^k$ and $x_j^0\neq 0$. Furthermore, two collections of such $\cx^k$-valued polynomials give the same $p(\epsilon)$ if and only if they differ by an element of the jet group
$$ \Bigl\{\bigl(\lambda_1(\epsilon),\dots,\lambda_b(\epsilon)\bigr)\,;\, \prod_{j=1}^b\lambda_j(\epsilon)=1\Bigr\},$$
where each $\lambda_j(\epsilon)$ is a polynomial of degree at most $k-1$. This group acts by multiplying $x_j(\epsilon)$ on the right.
Since we assume that $D$ is nondegenerate, the vectors $x_j^m$, $m=0,\dots, k-1$, are linearly independent. If we write $g_j$ for the invertible matrix with columns  $x_j^m$, then action $x_j(\epsilon)\mapsto x_j(\epsilon)\lambda(\epsilon)$ followed by truncation corresponds to the right multiplication of $g_i$ by $\lambda(J)$, where $J$ is the Jordan block with ones above the diagonal. Therefore the map $$D\mapsto \bigl(J,[g_1,\dots,g_b])\in \bigl(\mathscr{S}_{\gl(k,\cx)}\times GL(k,\cx)^b\bigr)/A_0$$
is well-defined and a biholomorphism between $\pi^{-1}(k0)$ and the subset of $U^{b,0}_k$ where $S=J$.
\par
Now consider a general $D$, which can be written as $D=\sum D_i$ with $\pi(D_i)=k_iz_i$. In other words $D$ is the zero set of the polynomial $q(z)=\prod(z-z_i)^{k_i}$ which corresponds to an element $S$ of $\mathscr{S}_{\gl(k,\cx)}$. Let $J(S)$ be a Jordan normal form of $S$.
Since the stabiliser of $D$ is trivial in each factor $G_\cx$, we have a decomposition $\cx^k\simeq\bigoplus V_{ji}$, $\dim V_{ji}=k_i$, for each $j=1,\dots,b$, such that $D_i$ is a point in $\Hilb^{k_i}_\pi(Y_i)$, where $Y_i=\cx \times \prod_{j=1}^b V_{ji}$. For each $D_i$ we obtain, owing to the special case above, an element $G_i$ of $\prod_{j=1}^b GL(V_{ji})/A_{0,i}$, where $A_{0,i}$ acts as in \eqref{P-action} with $X$ being the corresponding Jordan block of $J(S)$. These $G_i$ combine to give an element $G$ of $GL(k,\cx)^b/A_0$. Moreover, the choice of a Jordan normal form corresponds precisely to an ordered choice 
$(z-z_{i_1})^{k_{i_1}},\dots,(z-z_{i_s})^{k_{i_s}}$ of factors of $q(z)$. Thus we obtain a $GL(k,\cx)^{b}$-equivariant
holomorphic bijection from $U^{\rm non-deg}\subset\Hilb^k_\pi(X_{k}^{b})$ to $U^{b}_k$. Since both spaces are normal, Zariski's main theorem implies that this map is a biholomorphism.
\end{proof}
%%%%%%
\begin{example} For $b=1$, $ \Hilb^k_\pi(X_{k}^{1,0})$ is the space of based rational maps from $\oP^1$ to $\oP^k$ of degree $k$ \cite[p.53]{AH}. $U^{1,0}_k$ is then the subset of based {\em full} rational maps \cite{Craw}. 
\end{example}
\begin{remark} It is perhaps worth pointing out that $W^{1,1}_k\simeq T^\ast GL(k,\cx)$ is the space of certain based rational maps from $\oP^1$ to $\Gr_k(\cx^{2k})$ \cite{Masz}. Thus $W^{1,0}_k, W^{0,1}_k$, and $W^{1,1}_k$ are certain strata of charge $k$ moduli spaces of Euclidean monopoles with minimal symmetry breaking, with gauge group $SU(3)$ in the case of $W^{1,0}_k$ and $W^{0,1}_k$, and $SU(4)$ in the case of $W^{1,1}_k$. Is there a similar interpretation of $W^{2,1}_k$ and $W^{1,2}_k$ as strata of charge $k$ $SU(5)$-monopoles with minimal symmetry breaking?
\end{remark}
%%%%%%
\begin{remark} It follows that $U^{b,b^\prime}_{SL(k,\cx)}$ is isomorphic to a subset of $U^{b,b^\prime}_{GL(k,\cx)}$ 
consisting of ``centred" $D$ (i.e.\ if $\pi(D)=\sum k_iz_i$ as a divisor, then $\sum k_iz_i=0$ as the sum of complex numbers),
and such that the element $G=[g_1,\dots,g_{b+b^\prime}]$ of $GL(k,\cx)^{b+b^\prime}/A_0$, obtained in the above proof, satisfies $\prod\det g_i=1$.
\par
Observe also that $U^{b,b^\prime}_{PGL(k,\cx)}$ is a symplectic quotient of $U^{b,b^\prime}_{GL(k,\cx)}$ by
by the diagonal action of $\cx^\ast=\bigl\{(\lambda I,\dots,\lambda I)\in GL(k,\cx)^{b+b^\prime}\bigr\}$. 
\end{remark}

\subsection{The Fitting-transverse Hilbert scheme}

We now consider a larger subset of $\Hilb^k(X_{k}^{b,b^\prime})$,  consisting of subschemes $D$ such that  its Fitting image $\pi_{\rm Fitt}(D)$ has length $k$. This means that if $D$ has locally length $l$ and is set-theoretically supported in a single fibre $\pi^{-1}(z)$, then (locally) $\pi_\ast\sO_D\simeq \bigoplus_{i=1}^r \sO_{l_iz}$ with $\sum_{i=1}^r l_i=l$. Yet another way of saying this is that $D=\bigsqcup_{i=1}^t D_i$ with each $D_i$ (set-theoretically) supported at a single point and transverse to $\pi$. Such subschemes are l.c.i., and therefore this open subset is smooth. We define  $F_k^{b,b^\prime}$ to be the subset of this open set consisting of {\em locally nondegenerate} subschemes, i.e. such that the stabiliser of $D$ in each factor   of $\prod_{i=1}^{b+b^\prime} GL(k,\cx)$  is finite.
%%%
\par
We shall now show that $F^{1,0}_k$ (and of course $F_k^{0,1}$) is a Hamiltonian presymplectic manifold, i.e.\ it has a closed $2$-form and a $G$-equivariant moment map. For an introduction to presymplectic manifolds see \cite{Bot}.
%%%%%%
\begin{proposition} The symplectic form on $W_k^{1,0}$ extends to a closed form on $F_k^{1,0}$. The action of $GL(k,\cx)$ is locally free and Hamiltonian, and the moment map induces a bijection between $GL(k,\cx)$-orbits in $F_k^{1,0}$ and adjoint orbits in $\gl(k,\cx)$.\label{F10}
\end{proposition}
\begin{proof} 
Let $D\in F_k^{1,0}$. We can write it as $D=\bigcup_{i=1}^s  D_i$ with $\pi_\ast\sO_{D_i}=\bigoplus_{j=1}^{r_i} \sO_{l_{ij} z_i}$, where $z_i\in \cx$ are distinct and $\sum_{ij}l_{ij}=k$.  We can associate to $D$ a $k\times k$ matrix
in Jordan normal form:
\begin{equation} J(D)=\bigoplus_{i=1}^s\bigoplus_{j=1}^{r_i} J_{z_i,l_{ij}}.\label{Jordan}\end{equation}
 Moreover we can associate to $D$ a collection $(u_1,\dots,u_k)$ of $k$ nonzero vectors in $\cx^k$, determined up to the subgroup of $\Sym_k$ permuting identical Jordan blocks, and obtained as in the proof of Theorem \ref{tHsS}, and arranged in the order corresponding to \eqref{Jordan}. Since $D$ is locally nondegenerate, $G(D)=(u_1,\dots,u_k)$ is an invertible matrix.
 Observe that any $D\in F_k^{1,0}$ has a neighbourhood where $J(D)$ does not acquire additional nilpotent parts, and, therefore, 
the map  $\mu:F_k^{1,0}\to \gl(k,\cx)$, $\mu(D)=G(D)J(D)G(D)^{-1}$, is well-defined and holomorphic.
\par
Local coordinates near $D$ are given by the diagonal part of $J(D)$ and  $G(D)$. Therefore $T_DF_k^{1,0}$ is the direct sum of its 
subspace $\g^{\rm v}$ spanned by fundamental vector fields, and a subspace $P$ isomorphic to the tangent space to the diagonal matrices at $J(D)$. Let $v,v^\prime \in TF^{1,0}_k$ which we decompose as $v=\rho+x$, $v^\prime=\rho^\prime + x^\prime$, with $\rho,\rho^\prime \in \g^{\rm v}$, $x,x^\prime \in P$. The form
$$\omega(v,v^\prime)=\langle \rho,d\mu(v^\prime)\rangle - \langle \rho^\prime,d\mu(v)\rangle$$
is a closed $2$-form extending \eqref{form1}. The map $\mu$ is then a moment map for $\omega$ and the second statement is automatic.
\end{proof}
%%%
\begin{remark} It is perhaps worth pointing out that the $G^\cx$-action on $F_k^{1,0}$ is not proper. Indeed, the quotient is the non-Hausdorff space of all adjoint orbits in $\gl(k,\cx)$.
\end{remark}
\begin{remark}  The above proof also indicates how to construct an analogous space $F^{1,0}_{G_\cx}$ for any reductive Lie group by attaching $G_\cx$ to adjoint orbits in certain combinatorial way. We leave the details to the reader.\label{anyG}
\end{remark}
\begin{remark} The closed form $\omega$ on $F_k^{1,0}$ can be written locally in the form \eqref{form1} where $X=J(D)$.
\label{section}
\end{remark}
%%%
The manifolds $F_k^{b,b^\prime}$ are presymplectic and Hamiltonian for any $b,b^\prime$. However $F_k^{1,1}\not\simeq T^\ast GL(k,\cx)$ and $F_k^{b,b^\prime}$ do not seem  well behaved with respect to symplectic quotients (even apart from the fact that the action is not proper).

\end{document}